\documentclass{amsart}

\usepackage[shortlabels]{enumitem}
\usepackage{color} 
\usepackage{tikz} 
\usetikzlibrary{matrix} 
\usepackage{hyperref}
\usepackage {cleveref}

\newtheorem{theorem}{Theorem}[section]
\newtheorem{proposition}[theorem]{Proposition}
\newtheorem{lemma}[theorem]{Lemma}
\newtheorem{corollary}[theorem]{Corollary}

\theoremstyle{definition}
\newtheorem{definition}[theorem]{Definition}

\newtheorem{problem}[theorem]{Problem}

\theoremstyle{remark}
\newtheorem{remark}[theorem]{Remark}
\newtheorem{example}[theorem]{Example}

\numberwithin{equation}{section}

\begin{document}
\title[Standard Special Generic Maps of Homotopy Spheres]{Standard Special Generic Maps of Homotopy Spheres into Euclidean Spaces}

\author[D.J. Wrazidlo]{Dominik J. Wrazidlo}
\address{Mathematisches Institut, Ruprecht-Karls-Universit\"{a}t Heidelberg, Im Neuenheimer Feld 205, 69120 Heidelberg, Germany}
\email{dwrazidlo@mathi.uni-heidelberg.de}
\thanks{This work was supported by a scholarship for doctoral research awarded by the German National Merit Foundation.}


\date{\today.}


\begin{abstract}
A so-called special generic map is by definition a map of smooth manifolds all of whose singularities are definite fold points.
It is in general an open problem posed by Saeki in 1993 to determine the set of integers $p$ for which a given homotopy sphere admits a special generic map into $\mathbb{R}^{p}$. \par
By means of the technique of Stein factorization we introduce and study certain special generic maps of homotopy spheres into Euclidean spaces called \emph{standard}.
Modifying a construction due to Weiss, we show that standard special generic maps give naturally rise to a filtration of the group of homotopy spheres by subgroups that is strongly related to the Gromoll filtration.
Finally, we apply our result to concrete homotopy spheres, which particularly answers Saeki's problem for the Milnor $7$-sphere.
\end{abstract}

\maketitle

\section{Introduction}

A smooth map $f$ between smooth manifolds is traditionally called a \emph{special generic map} if every singular point $x$ of $f$ is a definite fold point, i.e., $f$ looks in suitable charts around $x$ and $f(x)$ like the multiple suspension of a positive definite quadratic form (see \Cref{subsection special generic maps}).

For a closed smooth manifold $M^{n}$ of dimension $n$, let $S(M^{n})$ denote the set of all integers $p \in \{1, \dots, n\}$ for which there exists a special generic map $M^{n} \rightarrow \mathbb{R}^{p}$.
Note that $S$ is a diffeomorphism invariant of smooth manifolds.
In 1993, the following problem was posed by Saeki in \cite[Problem 5.3, p. 177]{S2} (see also \cite{S3}).

\begin{problem}\label{study special generic maps}
Study the set $S(M^{n})$.
\end{problem}

For orientable $M^{n}$, Eliashberg \cite{E} showed that $n \in S(M^{n})$ if and only if $M^{n}$ is stably parallelizable.
We are concerned with the case that $1 \in S(M^{n})$, i.e., the case that $M^{n}$ admits a special generic map into $\mathbb{R}$.
Such a map is usually referred to as \emph{special generic function}, and is nothing but a Morse function all of whose critical points are extrema.
If $M^{n}$ admits a special generic function, then every component of $M^{n}$ is homeomorphic to $S^{n}$ by a well-known theorem of Reeb \cite{R}, and is for $n \leq 6$ even known to be diffeomorphic to $S^{n}$.
If $\Sigma^{n}$ denotes an exotic sphere of dimension $n \geq 7$, then according to \cite[(5.3.4), p. 177]{S2} (compare \Cref{remark on study of sgm}) we have
\begin{align}\label{sgm on exotic spheres}
\{1, 2, n\} \subset S(\Sigma^{n}) \subset \{1, 2, \dots, n-4, n\}.
\end{align}

Special generic functions can also prove useful in the study of individual homotopy spheres.
Indeed, a homotopy sphere is said to have \emph{Morse perfection} $\geq p$ (see \Cref{Morse perfection}) if it admits a family of special generic functions smoothly parametrized by points of the unit $p$-sphere that is subject to an additional symmetry condition.
In \cite{W} it is shown that the notion of Morse perfection is related as follows to the celebrated \emph{Gromoll filtration} (see \Cref{The Gromoll filtration}) of a homotopy sphere.

\begin{theorem}\label{gromoll filtration and morse perfection}
If $\Sigma^{n}$ is a homotopy sphere of dimension $n \geq 7$, then
$$
(\text{Gromoll filtration of } \Sigma^{n})-1 \leq \text{Morse perfection of } \Sigma^{n}.
$$
\end{theorem}

As pointed out in \cite{W}, it is possible by means of algebraic $K$-theory to derive upper bounds for the Morse perfection of certain homotopy spheres in terms of the signature of parallelizable null-cobordisms.
Consequently, \Cref{gromoll filtration and morse perfection} allows to draw conclusions on the Gromoll filtration of concrete homotopy spheres such as Milnor spheres (compare \Cref{proposition application to exotic spheres}$(ii)$).
From the point of view of differential geometry the upper bounds for the Morse perfection imply (see \cite[p. 388]{W}) that certain homotopy spheres do not admit a Riemannian metric with sectional curvature ``pinched''  in the interval $(1/4, 1]$. Note that this is precisely the pinching condition of the classical sphere theorem due to Rauch, Berger and Klingenberg. \par\medskip

The present paper focuses on \emph{standard} special generic maps (see \Cref{definition standard special generic map}) by which we roughly mean special generic maps from homotopy spheres into Euclidean spaces that factorize nicely over the closed unit ball (the required technique of \emph{Stein factorization} is recapitulated in \Cref{Stein factorization}).
According to \Cref{composition of ssgm with orthogonal projections are ssgm} a homotopy sphere that admits a standard special generic map into $\mathbb{R}^{p}$ will also admit such maps into $\mathbb{R}^{1}, \dots, \mathbb{R}^{p-1}$.
The \emph{fold perfection} of a homotopy sphere $\Sigma^{n}$ (see \Cref{definition fold perfection}) is defined to be the greatest integer $p$ for which a standard special generic map $\Sigma^{n} \rightarrow \mathbb{R}^{p}$ exists.

In refinement of \Cref{gromoll filtration and morse perfection} our main result is the following

\begin{theorem}\label{MAIN THEOREM}
If $\Sigma^{n}$ is a homotopy sphere of dimension $n \geq 7$, then
\begin{enumerate}[$(i)$]
\item $\text{Gromoll filtration of } \Sigma^{n} \leq \text{fold perfection of } \Sigma^{n}$, and
\item $(\text{fold perfection of } \Sigma^{n})-1 \leq \text{Morse perfection of } \Sigma^{n}$.
\end{enumerate}
\end{theorem}

Compared to Morse perfection the notion of fold perfection has the natural algebraic advantage that it gives rise to a filtration of the group of homotopy spheres by sub\emph{groups} (see \Cref{fold filtration remark}). \par\medskip

The paper is organized as follows.
In \Cref{preliminaries} we explain in detail all relevant notions and techniques.
\Cref{Fold perfection of homotopy spheres} introduces standard special generic maps and studies some of their properties.
The proof of \Cref{MAIN THEOREM}, which will be given in \Cref{proof of main theorem}, is based on a modification of the original proof of \Cref{gromoll filtration and morse perfection}.
Finally, in \Cref{Application to exotic spheres}, we discuss various applications of \Cref{MAIN THEOREM}.
\Cref{proposition application to exotic spheres} exploits known results about the Gromoll filtration and the Morse perfection of concrete exotic spheres such as Milnor spheres to extract information about their fold perfection.
Furthermore, \Cref{milnor 7-sphere} and \Cref{gromoll corollary} focus on the impact to \Cref{study special generic maps}, which includes an answer to \Cref{study special generic maps} for the Milnor $7$-sphere.

\subsection*{Notation}
In the following $M^{n}$ will always denote a connected closed smooth manifold of dimension $n \geq 1$.
The symbol $\cong$ will either mean diffeomorphism of smooth manifolds or isomorphism of groups.
The singular locus of a smooth map $f$ between smooth manifolds will be denoted by $S(f)$.
Euclidean $p$-space $\mathbb{R}^{p}$ is always equipped with the Euclidean inner product $(u, v) \mapsto u \cdot v$, and $||u|| := \sqrt{u \cdot u}$ denotes the corresponding Euclidean norm.
Let $D^{p} = \{x \in \mathbb{R}^{p}; ||x|| \leq 1\}$ denote the closed unit ball in $\mathbb{R}^{p}$, and $S^{p-1} := \partial D^{p}$ the standard $(p-1)$-sphere.
The $m \times m$ unit matrix will be denoted by $\mathbb{I}_{m}$.
For $p_{1} > p_{2}$ let $\pi^{p_{1}}_{p_{2}} \colon \mathbb{R}^{p_{1}} \rightarrow \mathbb{R}^{p_{2}}$ denote the projection to the \emph{last} $p_{2}$ coordinates.

\subsection*{Acknowledgement}
The author is grateful to Professor Osamu Saeki for pointing out the important application to the Milnor $7$-sphere.

\section{Preliminaries}\label{preliminaries}
The purpose of the present section is to recall the concepts that play a role in the presentation of \Cref{MAIN THEOREM}.

\subsection{Special generic maps}\label{subsection special generic maps}
Consider a smooth map $f \colon M^{n} \rightarrow N^{p}$ between smooth manifolds of dimensions $n \geq p \geq 1$. A point $x \in M$ is called \emph{fold point} if there exist local coordinates $x_{1}, \dots, x_{n}$ around $x$ and $y_{1}, \dots, y_{p}$ around $y := f(x)$ in which $f$ takes the form
\begin{align*}
y_{i} \circ f &= x_{i}, \qquad i = 1, \dots, p-1, \\
y_{p} \circ f &= x_{p}^{2} + \dots + x_{p+\lambda-1}^{2} - x_{p+\lambda}^{2} - \dots - x_{n}^{2},
\end{align*}
for a suitable integer $\lambda \in \{0, \dots, n-p+1\}$.
The integer $\operatorname{max}\{\lambda, n-p+1-\lambda\}$ turns out to be independent of the choice of coordinates, and is called \emph{absolute index} of $f$ at $x$.
The map $f$ is called \emph{fold map} if every singular point of $f$ is a fold point.
It is easy to see that the singular locus $S(f)$ of a fold map $f$ is a $(p-1)$-dimensional submanifold of $M^{n}$, and that $f$ restricts to an immersion $S(f) \rightarrow N^{p}$.
The absolute index is known to be a locally constant function on $S(f)$.

\begin{definition}\label{special generic maps and functions}
The smooth map $f \colon M^{n} \rightarrow N^{p}$ is called \emph{special generic map} if its singular locus $S(f)$ consists of \emph{definite} fold points, i.e., fold points of absolute index $n-p+1$. Special generic maps into $N^{p} = \mathbb{R}$ are referred to as \emph{special generic functions}.
\end{definition}

If $M^{n}$ admits a special generic function, then $M^{n}$ is homeomorphic to $S^{n}$ by a well-known result of Reeb \cite{R}, and for $n \leq 6$ even diffeomorphic to $S^{n}$ according to the proof of \cite[(5.3.3), p. 177]{S2}.

We mention the following folkloric method that allows to check in some cases that a given smooth map is a fold map.

\begin{proposition}\label{criterion for a map to be a fold map}
Let $n \geq q \geq 1$ be integers.
Let $f \colon U \rightarrow \mathbb{R}$ be a smooth function defined on an open subset $U \subset \mathbb{R}^{q-1} \times \mathbb{R}^{n-q+1}$.
Then the smooth map
$$
F \colon U \rightarrow \mathbb{R}^{q}, \qquad z = (x, y) \mapsto F(z) = (x, f(z)),
$$
is a fold map if and only if the Hessian
$$\textbf{H}_{y}(f_{x}) := \textbf{H}_{y}(f|_{U \cap (\{x\} \times \mathbb{R}^{n-q+1})})$$
is non-degenerate for every point $z = (x, y) \in S(f) = \{z \in U| \partial_{y_{1}}f(z) = \dots = \partial_{y_{n-q+1}}f(z) = 0\}$. In this case, the absolute index of $F$ at $z = (x, y) \in S(f)$ is given by $\operatorname{max}\{\lambda, n-p+1-\lambda\}$, where $\lambda$ denotes the number of negative diagonal entries of $\textbf{H}_{y}(f_{x})$ after diagonalization.
\end{proposition}

\begin{proof}
The result can be obtained by combining \cite[Proposition 3.3.4, p. 58]{W1} and \cite[Proposition 4.5.3, p. 103]{W1}.
\end{proof}

\begin{example}\label{special generic maps on the standard sphere}
The standard sphere $S^{n}$ admits a special generic map into $\mathbb{R}^{p}$ for all $p \in \{1, \dots, n\}$. In fact, the composition of the inclusion $S^{n} \subset \mathbb{R}^{n+1}$ with any orthogonal projection $\mathbb{R}^{n+1} \rightarrow \mathbb{R}^{p}$ is easily seen to be a special generic map.
\end{example}

\subsection{Stein factorization}\label{Stein factorization}

Special generic maps have first been studied by Burlet and de Rham in \cite{BdR} by means of the technique of Stein factorization, which has ever since played an essential role in the study of special generic maps, see e.g. \cite{S1}.

Let us recall the notion of Stein factorization of an arbitrary continuous map $f \colon X \rightarrow Y$ between topological spaces. Define an equivalence relation $\sim_{f}$ on $X$ as follows. Two points $x_{1}, x_{2} \in X$ are called equivalent, $x_{1} \sim_{f} x_{2}$, if they are mapped by $f$ to the same point $y := f(x_{1}) = f(x_{2}) \in Y$, and lie in the same connected component of $f^{-1}(y)$. The quotient map $q_{f} \colon X \rightarrow X/\sim_{f}$ gives rise to a unique set-theoretic factorization of $f$ of the form

\begin{center}
\begin{tikzpicture}
\path (0,0) node(a) {$X$}
         (3, 0) node(b) {$Y$}
         (0, -2) node(c) {$X/\sim_{f}$};
\draw[->] (a) -- node[above] {$f$} (b);
\draw[->] (a) -- node[left] {$q_{f}$} (c);
\draw[->] (c) -- node[below] {$\overline{f}$} (b);
\end{tikzpicture}.
\end{center}
If we equip $W_{f} := X/\sim_{f}$ with the quotient topology induced by the surjective map $q_{f} \colon X \rightarrow W_{f}$, then it follows that the maps $q_{f}$ and $\overline{f}$ are continuous. Then, the above diagram (and sometimes $W_{f}$ itself) is called the \emph{Stein factorization} of $f$.

For our purposes, the relevant properties of the Stein factorization of a special generic map can be summarized according to \cite[p. 267]{S1} in the following

\begin{proposition}\label{stein factorization}
Let $f \colon M^{n} \rightarrow \mathbb{R}^{p}$ be a special generic map where $p < n$. Then,
\begin{enumerate}[$(i)$]
\item the quotient space $W_{f} = M^{n}/\sim_{f}$ can be equipped with the structure of a smooth $p$-dimensional manifold with boundary in such a way that $q_{f} \colon M^{n} \rightarrow W_{f}$ is a smooth map satisfying $q_{f}^{-1}(\partial W_{f}) = S(f)$, and $\overline{f} \colon W_{f} \rightarrow \mathbb{R}^{p}$ is an immersion.
\item the quotient map $q_{f}$ restricts to a diffeomorphism $S(f) \stackrel{\cong}{\longrightarrow} \partial W_{f}$.
\end{enumerate}
\end{proposition}

\subsection{Morse perfection}\label{Morse perfection}
The concept of Morse perfection has been introduced in \cite[Definition 0.1, p. 387]{W}.

\begin{definition}\label{definition morse perfection}
The \emph{Morse perfection} of $M^{n}$ is the greatest integer $k \geq -1$ for which there exists a smooth map $\eta \colon S^{k} \times M^{n} \rightarrow \mathbb{R}$ with the following properties:
\begin{enumerate}[$(1)$]
\item $\eta$ restricts for every $s \in S^{k}$ to a special generic function (see \Cref{special generic maps and functions})
$$\eta_{s} \colon M^{n} \rightarrow \mathbb{R}, \qquad \eta_{s}(x) = \eta(s, x).$$
\item $\eta_{-s} = -\eta_{s}$ for all $s \in S^{k}$.
\end{enumerate}
\end{definition}

Note that every $M^{n}$ has Morse perfection $\geq -1$.
(In fact, $(1)$ and $(2)$ are empty conditions for $k=-1$ because $S^{-1} = \emptyset$ by convention.)

As pointed out in \cite{W}, the Morse perfection of $M^{n}$ is always $\leq n$, and the standard sphere $S^{n}$ has Morse perfection $n$.

\subsection{Homotopy spheres}\label{homotopy spheres}
Let $\Theta_{n}$ denote the group of $h$-cobordism classes of oriented homotopy $n$-spheres as introduced in \cite{KM}.
It is known that $\Theta_{n}$ is trivial for $n \leq 6$.
(For $n = 3$ this follows from the classical Poincar\'{e} conjecture proven by Perelman.)
For $n \geq 5$ any homotopy $n$-sphere is homeomorphic to $S^{n}$, and the equivalence relation of $h$-cobordism coincides with that of orientation preserving diffeomorphism.
This interpretation of $\Theta_{n}$ will be understood in the following.

Consider the group $\Gamma^{n} = \pi_{0}(\operatorname{Diff}(D^{n-1}, \partial))$ of isotopy classes of orientation preserving diffeomorphisms of $D^{n-1}$ that are the identity near the boundary.
For $n \geq 7$ a group isomorphism
$$\Sigma \colon \Gamma^{n} \stackrel{\cong}{\longrightarrow} \Theta_{n}$$
due to Smale and Cerf (see \cite{CS}) can explicitly be described as follows.
Let $\iota_{\pm} \colon D^{n-1} \rightarrow S^{n-1}$ denote the embedding $x \mapsto \iota_{\pm}(x) = (\pm\sqrt{1-||x||^{2}}, x)$.
A representative $g \colon D^{n-1} \stackrel{\cong}{\longrightarrow} D^{n-1}$ of a given $\gamma \in \Gamma^{n}$ induces a diffeomorphism $\varphi_{g} \colon S^{n-1} \stackrel{\cong}{\longrightarrow} S^{n-1}$ uniquely determined by requiring that $\varphi_{g} \circ \iota_{-} = \iota_{-} \circ g$ and $\varphi_{g} \circ \iota_{+} = \iota_{+}$.
(In particular, $\varphi_{g}$ agrees with the identity map in a neighborhood of the hemisphere of $S^{n-1}$ with non-negative first component.)
Then, $\Sigma(\gamma) \in \Theta_{n}$ is represented by the homotopy $n$-sphere that is obtained by identifying two copies of $D^{n}$ along their boundaries via $\varphi_{g}$.

\subsection{The Gromoll filtration}\label{The Gromoll filtration}
Let $n \geq 7$ be an integer.
There is a filtration of the group $\Gamma^{n} = \pi_{0}(\operatorname{Diff}(D^{n-1}, \partial))$ due to Gromoll \cite{G} by subgroups of the form
$$
0 = \Gamma^{n}_{n-1} \subset \dots \subset \Gamma^{n}_{1} = \Gamma^{n}.
$$
By definition, $\gamma \in \Gamma^{n}_{p+1}$ if $\gamma$ can be represented by a diffeomorphism $g \in \operatorname{Diff}(D^{n-1}, \partial)$ such that the following diagram commutes.
\begin{center}
\begin{tikzpicture}
\path (0,0) node(a) {$D^{n-1}$}
         (4, 0) node(b) {$D^{n-1}$}
         (0, -1) node(c) {$\mathbb{R}^{p}$}
         (4, -1) node(d) {$\mathbb{R}^{p}$};
\draw[->] (a) -- node[above] {$g$} (b);
\draw[->] (c) -- node[above] {$\operatorname{id}_{\mathbb{R}^{p}}$} (d);
\draw[->] (a) -- node[left] {$\pi^{n-1}_{p}|_{D^{n-1}}$} (c);
\draw[->] (b) -- node[right] {$\pi^{n-1}_{p}|_{D^{n-1}}$} (d);
\end{tikzpicture}
\end{center}
One says that $\gamma \in \Gamma^{n}$ (or $\Sigma(\gamma) \in \Theta_{n}$) has Gromoll filtration $p$ if $\gamma \in \Gamma^{n}_{p} \setminus \Gamma^{n}_{p+1}$.

\section{Fold perfection of homotopy spheres}\label{Fold perfection of homotopy spheres}

The following result is due to Saeki \cite[Proposition 4.1, p. 274]{S1}. 

\begin{proposition}\label{contractibility of stein factorization}
Let $f \colon M^{n} \rightarrow \mathbb{R}^{p}$ ($p < n$) be a special generic map.
Then $M^{n}$ is a homotopy sphere if and only if the Stein factorization $W_{f}$ is contractible.
\end{proposition}

\begin{corollary}\label{homology sphere as singular locus}
If $f \colon \Sigma^{n} \rightarrow \mathbb{R}^{p}$ ($p < n$) is a special generic map on a homotopy sphere $\Sigma^{n}$, then the singular locus $S(f)$ is a homology $(p-1)$-sphere.
\end{corollary}

\begin{proof}
Note that $q_{f}$ restricts by \Cref{stein factorization}$(ii)$ to a diffeomorphism $S(f) \cong \partial W_{f}$.
Moreover, \Cref{contractibility of stein factorization} implies that $W_{f}$ is contractible.
Hence, the claim follows from Poincar\'{e} duality, $H_{\ast}(W_{f}, \partial W_{f}; \mathbb{Z}) \cong H^{n-\ast}(W_{f}; \mathbb{Z})$ as well as the long exact sequence for reduced integral homology of the pair $(W_{f}, \partial W_{f})$.
\end{proof}

However, in the situation of \Cref{homology sphere as singular locus}, $W_{f}$ is in general not diffeomorphic to $D^{p}$ because $S(f) \cong \partial W_{f}$ is in general not even a homotopy sphere as remarked in \cite[Remark 4.4, p. 275]{S1}.

\begin{definition}\label{definition standard special generic map}
A special generic map $f \colon \Sigma^{n} \rightarrow \mathbb{R}^{p}$ ($p < n$) on a homotopy sphere $\Sigma^{n}$ is called \emph{standard} if the Stein factorization $W_{f}$ is diffeomorphic to $D^{p}$.
\end{definition}

It is convenient to clarify in dependence of the dimension $p$ what it means for a special generic map to be standard.

\begin{proposition}\label{characterization of standard special generic maps}
Suppose that $f \colon \Sigma^{n} \rightarrow \mathbb{R}^{p}$ ($p<n$) is a special generic map on a homotopy sphere $\Sigma^{n}$.
Then the following statements hold:
\begin{enumerate}[$(a)$]
\item If $p \in \{1, 2, 3\}$, then $f$ is a standard special generic map.
\item For $6 \leq p < n$ the following statements are equivalent:
\begin{enumerate}[$(i)$]
\item The map $f$ is a standard special generic map.
\item The singular locus $S(f)$ is diffeomorphic to $S^{p-1}$.
\item The singular locus $S(f)$ is simply connected.
\end{enumerate}
\item For $p = 5$ the equivalence $(i) \Leftrightarrow (ii)$ from part $(b)$ is valid.
\end{enumerate}
\end{proposition}

\begin{proof}
The claim of part $(a)$ follows for $p \in \{1, 2\}$ from the well-known classification of compact smooth manifolds of dimension $p$.
For $p=3$ the contractible compact smooth $3$-manifold $W_{f}$ (see \Cref{contractibility of stein factorization}) is (according to \Cref{homology sphere as singular locus}) bounded by a homology $2$-sphere $\partial W_{f} \cong S(f)$, which is necessarily diffeomorphic to $S^{2}$.
Hence the claim follows from a standard application of the $3$-dimensional version of the $h$-cobordism theorem.
(The latter is by \cite[p. 113]{M1} a consequence of the classical Poincar\'{e} conjecture proven by Perelman.)
Concerning part $(b)$, the implication $(i) \Rightarrow (ii)$ holds because $q_{f}$ restricts by \Cref{stein factorization}$(ii)$ to a diffeomorphism $S(f) \cong \partial W_{f}$, and $\partial W_{f} \cong S^{p-1}$. Moreover, $(ii) \Rightarrow (iii)$ as $p > 2$. In order to show $(iii) \Rightarrow (i)$, first note that $W_{f}$ is by \Cref{contractibility of stein factorization} a contractible compact smooth manifold with simply connected boundary $S(f) \cong \partial W_{f}$ by $(iii)$. As $W_{f}$ is of dimension $p \geq 6$, the claim is a direct consequence of the $h$-cobordism theorem (see \cite[Proposition A, p. 108]{M1}).
Finally, part $(c)$ is covered by \cite[Proposition C 1), p. 110f]{M1}.
\end{proof}

The following is a sufficient criterion for a special generic map to be standard.

\begin{lemma}\label{sufficient condition for standardness}
Let $f \colon M^{n} \rightarrow \mathbb{R}^{p}$ be a special generic map. If $f(M^{n}) = D^{p}$ and $\overline{f} \colon W_{f} \rightarrow \mathbb{R}^{p}$ is injective, then $M^{n}$ is a homotopy sphere and $f$ is standard.
\end{lemma}

\begin{proof}
Since $\overline{f}(W_{f}) = f(M^{n}) = D^{p}$, it suffices by \Cref{contractibility of stein factorization} and \Cref{definition standard special generic map} to note that the immersion $\overline{f}$ (see \Cref{stein factorization}$(i)$) is an embedding. This is true because $\overline{f}$ is injective by assumption and the domain $W_{f} = q_{f}(M^{n})$ is compact, so that the immersion $\overline{f}$ restricts to a homeomorphism onto its image.
\end{proof}

\begin{example}\label{ssgm on standard sphere}
\Cref{special generic maps on the standard sphere} and \Cref{sufficient condition for standardness} imply that any orthogonal projection $\mathbb{R}^{n+1} \rightarrow \mathbb{R}^{p}$ restricts to a standard special generic map on the standard sphere $S^{n}$ for all $p \in \{1, \dots, n-1\}$.
\end{example}

\begin{proposition}\label{composition of ssgm with orthogonal projections to the line are ssf}
Let $f \colon \Sigma^{n} \rightarrow \mathbb{R}^{p}$ be a standard special generic map.
Then, for every $u \in S^{p-1}$, the composition of the standard special generic map
$$
h \colon \Sigma^{n} \stackrel{q_{f}}{\longrightarrow} W_{f} \stackrel{\cong}{\longrightarrow} D^{p} \hookrightarrow \mathbb{R}^{p}
$$
(where a diffeomorphism $W_{f} \stackrel{\cong}{\longrightarrow} D^{p}$ has been fixed) with the orthogonal projection
$$
\pi_{u} \colon \mathbb{R}^{p} \rightarrow \mathbb{R}, \qquad v \mapsto u \cdot v,
$$
yields a special generic function $\pi_{u} \circ h \colon \Sigma^{n} \rightarrow \mathbb{R}$ with $S(\pi \circ h) = h^{-1}(\{\pm u\}) \cong S^{0}$.
\end{proposition}

\begin{proof}
\Cref{stein factorization} and \Cref{sufficient condition for standardness} imply that $h$ is a standard special generic map with image $h(\Sigma^{n}) = D^{p} \subset \mathbb{R}^{p}$ and singular locus $S(h) = h^{-1}(S^{p-1})$.
Also note that $h$ restricts to a diffeomorphism $S(h) \stackrel{\cong}{\longrightarrow} S^{p-1}$.
Moreover, $\pi := \pi_{u}$ restricts by \Cref{special generic maps on the standard sphere} to a special generic function $S^{p-1} \rightarrow \mathbb{R}$ with singular locus $S(\pi|_{S^{p-1}}) = \pi^{-1}(S^{0}) = \{\pm u\}$.
Altogether, $S(\pi \circ h) = h^{-1}(\{\pm u\}) \cong S^{0}$.
It suffices to show that both critical points of $\pi \circ h$ are non-degenerate.

As every $c \in S(\pi \circ h)$ is also critical point of $h$, there exist by \Cref{special generic maps and functions} charts $\varphi \colon U \rightarrow U' \subset \mathbb{R}^{n} = \mathbb{R}^{p-1} \times \mathbb{R}^{n-p+1}$ and $\psi \colon V \rightarrow V' \subset \mathbb{R}^{p} = \mathbb{R}^{p-1} \times \mathbb{R}$ with $h(U) \subset V$, $c \in U$, $\varphi(c) = (0, 0)$, and such that
$$
h' := \psi \circ h \circ \varphi^{-1} \colon U' \rightarrow V', \qquad h'(x, y) = (x, ||y||^{2}).
$$
Set $\pi' := \pi \circ \psi^{-1} \colon V' \rightarrow \mathbb{R}$.
A short computation shows that the Hessian of $\pi' \circ h' = \pi \circ h \circ \varphi^{-1}$ at $(x, y) = (0, 0) = \varphi(c) \in U' \subset \mathbb{R}^{p-1} \times \mathbb{R}^{n-p+1}$ is given by a block matrix of the form
$$
\textbf{H}_{(0, 0)}(\pi' \circ h') = \left(\begin{matrix}
\textbf{H}_{(0, 0)}(\pi' \circ h'|_{U' \cap (\mathbb{R}^{p-1}\times 0)}) & 0 \\
0 & 2 \cdot (\partial_{p}\pi')(0, 0) \cdot \mathbb{I}_{n-p+1}
\end{matrix}\right).
$$

To establish that the Hessian $\textbf{H}_{(0, 0)}(\pi' \circ h'|_{U' \cap (\mathbb{R}^{p-1}\times 0)})$ is non-singular note that $\pi' \circ h'|_{U' \cap (\mathbb{R}^{p-1}\times 0)} = \pi \circ (h \circ \varphi^{-1})|_{U' \cap (\mathbb{R}^{p-1}\times 0)}$ is the composition of the embedding $h \circ \varphi^{-1}| \colon U' \cap (\mathbb{R}^{p-1}\times 0) = S(h') = S(h \circ \varphi^{-1}) \rightarrow S^{p-1}$ with the Morse function $\pi|_{S^{p-1}} \colon S^{p-1} \rightarrow \mathbb{R}$, and that $(h \circ \varphi^{-1})(0, 0) = h(c) \in \{\pm u\} = S(\pi|_{S^{p-1}})$.

It remains to show that $2 \cdot (\partial_{p}\pi')(0, 0) \cdot \mathbb{I}_{n-p+1}$ is non-singular, i.e., $(\partial_{p}\pi')(0, 0) \neq 0$.
For this purpose, note that $\varphi(c) = (0, 0)$ is a critical point of $\pi' \circ h' = (\pi \circ h) \circ \varphi^{-1}$ because $c$ is a critical point of $\pi \circ h$ and $\varphi$ is a diffeomorphism.
Hence, by the chain rule,
\begin{align*}
0 &= J(\pi' \circ h', (0, 0)) = J(\pi', h'(0, 0)) \cdot J(h', (0, 0)) \\
&= \left(\begin{matrix}
(\partial_{1}\pi')(0, 0) & \dots & (\partial_{p-1}\pi')(0, 0) & (\partial_{p}\pi')(0, 0) \\
\end{matrix}\right) \cdot \left(\begin{matrix}
\mathbb{I}_{p-1} & 0 & \dots & 0 \\
0 & 0 & \dots & 0
\end{matrix}\right) \\
&= \left(\begin{matrix}
(\partial_{1}\pi')(0, 0) & \dots & (\partial_{p-1}\pi')(0, 0) & 0 & \dots & 0
\end{matrix}\right).
\end{align*}
Thus, $(\partial_{1}\pi')(0, 0) = \dots = (\partial_{p-1}\pi')(0, 0) = 0$.
Consequently, $(\partial_{p}\pi')(0, 0) \neq 0$ as $\pi' \colon V' \rightarrow \mathbb{R}$ is a submersion.
\end{proof}

\begin{corollary}\label{composition of ssgm with orthogonal projections are ssgm}
Let $f \colon \Sigma^{n} \rightarrow \mathbb{R}^{p}$ be a standard special generic map.
Then, for every $q \in \{1, \dots, p-1\}$, the composition of the standard special generic map
$$
h \colon \Sigma^{n} \stackrel{q_{f}}{\longrightarrow} W_{f} \stackrel{\cong}{\longrightarrow} D^{p} \hookrightarrow \mathbb{R}^{p}
$$
(where a diffeomorphism $W_{f} \stackrel{\cong}{\longrightarrow} D^{p}$ has been fixed) with the projection $\pi \colon \mathbb{R}^{p} \rightarrow \mathbb{R}^{q}$ to the first $q$ components yields a standard special generic map $\pi \circ h \colon \Sigma^{n} \rightarrow \mathbb{R}^{q}$.
\end{corollary}

\begin{proof}
It suffices to show that $g := \pi \circ h \colon \Sigma^{n} \rightarrow \mathbb{R}^{q}$ is a special generic map.
In fact, \Cref{sufficient condition for standardness} will then imply that $g$ is standard because $g(\Sigma^{n}) = \pi(D^{p}) = D^{q}$, and the fibers of $g$ are all connected.
(To show that $g$ has connected fibers, one uses that $q_{f}$ and $\pi| \colon D^{p} \rightarrow D^{q}$ are surjections with connected fibers between compact Hausdorff spaces.)

Analogously to the proof of \Cref{composition of ssgm with orthogonal projections to the line are ssf} one shows that $h$ is a standard special generic map, and that $S(g) = h^{-1}(S(\pi|_{S^{p-1}})) = h^{-1}(S^{q-1} \times 0) \cong S^{q-1}$.
Furthermore, observe that $g$ restricts to a diffeomorphism $S(g) \stackrel{\cong}{\longrightarrow} S^{q-1}$.

Let $c \in S(g)$ be a singular point of $g$, and let $u := g(c) \in S^{q-1}$.
Let $\pi_{u} \colon \mathbb{R}^{q} \rightarrow \mathbb{R}$ denote the orthogonal projection given by $v \mapsto u \cdot v$.
Let $\{u_{1}, \dots, u_{q-1}, u\}$ be an extension to an orthonormal basis of $\mathbb{R}^{q}$.
Let $\pi_{u}^{\perp} \colon \mathbb{R}^{q} \rightarrow \mathbb{R}^{q-1}$ denote the orthogonal projection given by $\lambda_{1} u_{1} + \dots + \lambda_{q-1} u_{q-1} + \lambda u \mapsto (\lambda_{1}, \dots, \lambda_{q-1})$.
One obtains a linear isomorphism $\psi := (\pi_{u}^{\perp}, \pi_{u}) \colon \mathbb{R}^{q} \rightarrow \mathbb{R}^{q}$.

The composition $\alpha := \pi_{u}^{\perp} \circ g|_{S(g)}$ is a local diffeomorphism at $c$ such that $\alpha(c) = 0 \in \mathbb{R}^{q-1}$.
Thus, $\alpha$ restricts to a chart $U_{S} \stackrel{\cong}{\longrightarrow} U_{S}'$ of $S(g) \cong S^{q-1}$ between suitable open neighborhoods $c \in U_{S} \subset S(g)$ and $0 \in U_{S}' \subset \mathbb{R}^{q-1}$.
Chosen appropriately, $U_{S}$ has a trivial tubular neighborhood $\nu \colon U_{S} \times \mathbb{R}^{n-q+1} \stackrel{\cong}{\longrightarrow} U_{0}$ in $\Sigma^{n}$.
As $\pi_{u}^{\perp} \circ g|_{U_{S}} = \alpha|_{U_{S}}$, the Jacobian of
$$\varphi := (\pi_{u}^{\perp} \circ g, \operatorname{pr}_{\mathbb{R}^{n-q+1}} \circ \nu^{-1}) \colon U_{0} \rightarrow \mathbb{R}^{q-1} \times \mathbb{R}^{n-q+1}$$
is invertible at points in $S(g) \cap U_{0} = \nu(U_{S} \times 0) = U_{S}$.
In particular, $\varphi$ restricts to a chart $U \rightarrow U'$ of $\Sigma^{n}$ between suitable open neighborhoods $c \in U \subset U_{0}$ and $\varphi(c) = (0, 0) \in U' \subset \mathbb{R}^{q-1} \times \mathbb{R}^{n-q+1}$.

Using that $\pi_{u}^{\perp}(g(\varphi^{-1}(x, y))) = x$ by construction of $\varphi$, consider the composition
$$
\psi \circ g \circ \varphi^{-1} \colon U' \rightarrow \mathbb{R}^{q}, \qquad (x, y) \mapsto (x, (\pi_{u} \circ g \circ \varphi^{-1})(x, y)).
$$
Observe that $\pi_{u} \circ g = \pi_{\tilde{u}} \circ h$, where $\tilde{u} := (u, 0) \in S^{p-1} \subset \mathbb{R}^{q} \times \mathbb{R}^{p-q}$.
Now \Cref{composition of ssgm with orthogonal projections to the line are ssf} implies that $\pi_{\tilde{u}} \circ h$ is a special generic function whose singular locus $S(\pi_{\tilde{u}} \circ h) = h^{-1}(\{\pm \tilde{u}\})$ contains $c = h^{-1}(\tilde{u})$.
Consequently, $(0, 0) = \varphi(c)$ is a critical point of $\pi_{u} \circ g\circ \varphi^{-1}$ such that the Hessian $\textbf{H}_{(0, 0)}(\pi_{u} \circ g \circ \varphi^{-1})$ is definite.
Therefore, $(0, 0)$ is also a critical point of $\pi_{u} \circ g\circ \varphi^{-1}|_{U' \cap (0 \times \mathbb{R}^{n-q+1})}$ such that the Hessian $\textbf{H}_{(0, 0)}(\pi_{u} \circ g \circ \varphi^{-1}|_{U' \cap (0 \times \mathbb{R}^{n-q+1})})$ is definite.
In conclusion, \Cref{criterion for a map to be a fold map} implies that $c$ is a definite fold point of $g$.
\end{proof}

In view of \Cref{composition of ssgm with orthogonal projections are ssgm} we define the notion of fold perfection of homotopy spheres.

\begin{definition}\label{definition fold perfection}
The \emph{fold perfection} of a homotopy sphere $\Sigma^{n}$ ($n \geq 7$) is the greatest integer $p \geq 1$ for which there exists a standard special generic map $\Sigma^{n} \rightarrow \mathbb{R}^{p}$.
\end{definition}

Note that by (\ref{sgm on exotic spheres}) and \Cref{ssgm on standard sphere} the standard sphere $S^{n}$ is the only homotopy $n$-sphere of fold perfection $n-1$.
By (\ref{sgm on exotic spheres}) and \Cref{characterization of standard special generic maps}$(a)$ any homotopy sphere $\Sigma^{n}$ has fold perfection $\geq 2$.
In \Cref{proposition application to exotic spheres}$(ii)$ we will see that the fold perfection of Milnor spheres equals $2$.
Furthermore, \Cref{proposition application to exotic spheres}$(i)$ shows that there exist exotic spheres with fold perfection greater than $2$.

\begin{remark}
The notion of fold perfection has only been defined for homotopy spheres of dimension $\geq 7$.
As discussed in \Cref{preliminaries}, this is due to the fact that in lower dimensions the standard sphere turns out to be the only homotopy sphere that admits a special generic function.
In fact, it is well-known that there are no homotopy $n$-spheres except for the standard sphere $S^{n}$ in dimension $n \leq 6$, $n \neq 4$.
\end{remark}

The following fact will notably be applied to the Milnor $7$-sphere in \Cref{milnor 7-sphere}.

\begin{proposition}\label{fold perfection 2}
If $\Sigma^{n}$ ($n \geq 7$) is a homotopy sphere of fold perfection $2$, then $3 \notin S(\Sigma^{n})$.
\end{proposition}

\begin{proof}
Any special generic map $\Sigma^{n} \rightarrow \mathbb{R}^{3}$ is standard by \Cref{characterization of standard special generic maps}$(a)$. Hence $3 \in S(\Sigma^{n})$ would imply that $\Sigma^{n}$ has fold perfection $\geq 3$.
\end{proof}

\begin{remark}\label{fold filtration remark}
By \Cref{composition of ssgm with orthogonal projections are ssgm} the notion of fold perfection gives rise to a filtration of $\Theta_{n}$ ($n \geq 7$) by subsets $F^{n}_{n-1} \subset \dots \subset F^{n}_{1} \subset \Theta_{n}$, where, by definition, $[\Sigma^{n}] \in F^{n}_{p}$ if the fold perfection of $\Sigma^{n}$ is $\geq p$.
The proof of \cite[Lemma 5.4, p. 278]{S1} implies that $F^{n}_{p}$ is in fact a filtration of $\Theta_{n}$ by sub\emph{groups}, which contains the Gromoll filtration $\Gamma^{n}_{p}$ by \Cref{MAIN THEOREM}$(i)$.
We do not know whether the groups $F^{n}_{p}$ and $\Gamma^{n}_{p}$ do in general coincide or not.
\end{remark}

\section{Proof of \Cref{MAIN THEOREM}}\label{proof of main theorem}

Let $\Sigma^{n}$ be a homotopy sphere of dimension $n \geq 7$. \par\medskip

The proof of part $(i)$ is a modification of the original proof due to Weiss (see \cite[\S 4, pp. 403 ff]{W}) of \Cref{gromoll filtration and morse perfection}.

According to \Cref{homotopy spheres} there exists a diffeomorphism $g \in \operatorname{Diff}(D^{n-1}, \partial)$ such that $\Sigma^{n}$ is diffeomorphic to the twisted sphere obtained by gluing the open balls
\begin{align*}
V_{-} := S^{n} \setminus \{(1, 0, \dots, 0)\}, \\
V_{+} := S^{n} \setminus \{(-1, 0, \dots, 0)\},
\end{align*}
along the open cylinder $V_{-} \cap V_{+} = S^{n} \setminus \{(\pm 1, 0, \dots, 0)\} \cong S^{n-1} \times (-1, 1) =: C$.
The result of the gluing procedure is described by a pushout diagram of the form
\begin{center}
\begin{tikzpicture}
\path (0,0) node(a) {$C$}
         (4, 0) node(b) {$V_{+}$}
         (0, -2) node(c) {$V_{-}$}
         (4, -2) node(d) {$\Sigma^{n}$};
\draw[->] (a) -- node[left] {$j_{-}$} (c);
\draw[->] (a) -- node[above] {$j_{+} \circ (\varphi_{g} \times \operatorname{id}_{(-1, 1)})$} (b);
\draw[->] (b) -- node[right] {$k_{+}$} (d);
\draw[->] (c) -- node[above] {$k_{-}$} (d);
\end{tikzpicture},
\end{center}
where the embeddings $j_{-} \colon C \rightarrow V_{-}$ and $j_{+} \colon C \rightarrow V_{+}$ are both given by the formula
$$
(z, t) \mapsto (\sqrt{1-t^{2}} \cdot z, t) \in V_{-} \cap V_{+} \subset \mathbb{R}^{n} \times \mathbb{R}.
$$

Supposing that $\Sigma^{n}$ has Gromoll filtration $> p$, we may assume by \Cref{The Gromoll filtration} that $g$ is chosen such that $\pi^{n-1}_{p}(g(x)) = \pi^{n-1}_{p}(x)$ for all $x \in D^{n-1}$.
For the associated diffeomorphism $\varphi_{g} \colon S^{n-1} \rightarrow S^{n-1}$ this implies $\pi^{n}_{p}(\varphi_{g}(z)) = \pi^{n}_{p}(z)$ for all $z \in S^{n-1}$.
(In fact, in terms of the embeddings $\iota_{\pm} \colon D^{n-1} \rightarrow S^{n-1}$ of \Cref{homotopy spheres} the diffeomorphism $\varphi_{g}$ satisfies for all $x \in D^{n-1}$ the equations
$$\pi^{n}_{p}(\varphi_{g}(\iota_{-}(x))) = \pi^{n}_{p}(\iota_{-}(g(x))) = \pi^{n}_{p}(-\sqrt{1-||g(x)||^{2}}, g(x)) = \pi^{n-1}_{p}(g(x)) = \pi^{n}_{p}(x)$$
as well as $\pi^{n}_{p}(\varphi_{g}(\iota_{+}(x))) = \pi^{n}_{p}(\iota_{+}(x)) = \pi^{n}_{p}(\sqrt{1-||x||^{2}}, x) = \pi^{n-1}_{p}(x)$.)

Consequently, the following diagram commutes.
\begin{center}
\begin{tikzpicture}

\path (0,0) node(a) {$C$}
         (4, 0) node(b) {$V_{+}$}
         (4, -1) node(c) {$\mathbb{R}^{n+1}$}
         (0, -2) node(d) {$V_{-}$}
         (2, -2) node(e) {$\mathbb{R}^{n+1}$}
         (4, -2) node(f) {$\mathbb{R}^{p+1}$};
\draw[->] (a) -- node[left] {$j_{-}$} (d);
\draw[->] (a) -- node[above] {$j_{+} \circ (\varphi_{g} \times \operatorname{id}_{(-1, 1)})$} (b);
\draw[->] (b) -- node[right] {$\operatorname{incl}$} (c);
\draw[->] (c) -- node[right] {$\pi^{n+1}_{p+1}$} (f);
\draw[->] (d) -- node[above] {$\operatorname{incl}$} (e);
\draw[->] (e) -- node[above] {$\pi^{n+1}_{p+1}$} (f);
\end{tikzpicture}
\end{center}
(In fact, for all $(z, t) \in C$ we have
\begin{align*}
&(\pi^{n+1}_{p+1} \circ j_{+} \circ (\varphi_{g} \times \operatorname{id}_{(-1, 1)}))(z, t) \\
&= \pi^{n+1}_{p+1}(j_{+}(\varphi_{g}(z), t)) = \pi^{n+1}_{p+1}(\sqrt{1-t^{2}} \cdot \varphi_{g}(z), t) \\
&= (\sqrt{1-t^{2}} \cdot \pi^{n}_{p}(\varphi_{g}(z)), t) = (\sqrt{1-t^{2}} \cdot \pi^{n}_{p}(z), t) \\
&= \pi^{n+1}_{p+1}(\sqrt{1-t^{2}} \cdot z, t) = (\pi^{n+1}_{p+1} \circ j_{-})(z, t).)
\end{align*}

Thus, the universal property of the above pushout diagram gives rise to a map $f \colon \Sigma^{n} \rightarrow \mathbb{R}^{p+1}$ such that $f \circ k_{\pm} = \pi^{n+1}_{p+1}|_{V_{\pm}}$.
With $k_{-}(V_{-}) \cup k_{+}(V_{+})$ being an open cover of $\Sigma^{n}$, \Cref{special generic maps on the standard sphere} and \Cref{sufficient condition for standardness} imply that $f$ is a standard special generic map.
In conclusion, $\Sigma^{n}$ has fold perfection $> p$.
\par\medskip

The proof of part $(ii)$ follows immediately from \Cref{composition of ssgm with orthogonal projections to the line are ssf}. In fact, suppose that $\Sigma^{n}$ has fold perfection $\geq p$. If $f \colon \Sigma^{n} \rightarrow \mathbb{R}^{p}$ is a standard special generic map and $\iota \colon W_{f} \cong D^{p}$ is a diffeomorphism, then the smooth map
$$
\eta \colon S^{p-1} \times \Sigma^{n} \rightarrow \mathbb{R}, \qquad \eta(u, x) = u \cdot \iota(q_{f}(x))
$$
satisfies properties $(1)$ and $(2)$ of \Cref{definition morse perfection}, which shows that $\Sigma^{n}$ has Morse perfection $\geq p-1$.
\par\medskip

This completes the proof of \Cref{MAIN THEOREM}.

\section{Applications}\label{Application to exotic spheres}

We start with an application of \Cref{MAIN THEOREM} to some concrete exotic spheres.

\begin{proposition}\label{proposition application to exotic spheres}
\begin{enumerate}[$(i)$]
\item In certain dimensions $n \geq 7$ the group $\Theta_{n}$ is known to contain exotic spheres of great depth in the Gromoll filtration. For instance, $\Gamma_{3}^{10} \neq 0$ and $\Gamma_{11}^{18} \neq 0$ by \cite[Appendix A]{CS}. Hence, these exotic spheres have at least an accordingly great fold perfection by \Cref{MAIN THEOREM}$(i)$.
\item Let $n = 4k-1$ for some integer $k \geq 2$, and let $\Sigma^{n}_{M}$ denote the \emph{Milnor $n$-sphere}, i.e., $\Sigma^{n}_{M} = \partial W^{n+1}$ for some parallelizable cobordism $W^{n}$ with signature $8$. By \cite[p. 390]{W} one has $(\text{Gromoll filtration of } \Sigma^{n}_{M})-1 = 1 = \text{Morse perfection of } \Sigma^{n}_{M}$.
Consequently, the fold perfection of $\Sigma^{n}_{M}$ is precisely $2$ by \Cref{MAIN THEOREM}.
\end{enumerate}
\end{proposition}

In view of (\ref{sgm on exotic spheres}), \Cref{proposition application to exotic spheres}$(ii)$ implies the following answer of \Cref{study special generic maps} for the Milnor $7$-sphere by invoking \Cref{fold perfection 2}.

\begin{corollary}\label{milnor 7-sphere}
The Milnor $7$-sphere $\Sigma^{7}_{M}$ of \Cref{proposition application to exotic spheres}$(ii)$ satisfies 
$$
S(\Sigma^{7}_{M}) = \{1, 2, 7\}.
$$
\end{corollary}

\begin{remark}
Since the subgroup $F_{3}^{7} \subset \Theta_{7}$ (see \Cref{fold filtration remark}) has at least index $2$ due to $[\Sigma^{7}_{M}] \notin F_{3}^{7}$, there are in fact at least $14$ exotic $7$-spheres $\Sigma^{7}$ with $S(\Sigma^{7}) = \{1, 2, 7\}$.
We do not know the actual size of the groups $\Gamma_{3}^{7} \subset F_{3}^{7}$.
\end{remark}

With regard to \Cref{study special generic maps} we obtain the following immediate consequence of \Cref{MAIN THEOREM}$(i)$ and \Cref{composition of ssgm with orthogonal projections are ssgm}.

\begin{proposition}\label{gromoll corollary}
If $\Sigma^{n}$ is a homotopy sphere of dimension $n \geq 7$ whose Gromoll filtration is $\geq k$, then
$$
\{1, \dots, k, n\} \subset S(\Sigma^{n}).
$$
\end{proposition}

\begin{remark}\label{remark on study of sgm}
Let us compare \Cref{gromoll corollary} with the inclusions of (\ref{sgm on exotic spheres}).
\begin{enumerate}[$(i)$]
\item \Cref{gromoll corollary} implies that a homotopy sphere $\Sigma^{n}$ ($n \geq 7$) with Gromoll filtration $\geq n-3$ must be diffeomorphic to $S^{n}$ because the second inclusion of (\ref{sgm on exotic spheres}) shows that $n-3 \notin S(\Sigma^{n})$ whenever $\Sigma^{n}$ is an exotic sphere.
Indeed, both of these facts are known to follow from Hatcher's proof \cite{H} of the Smale conjecture according to the introduction of \cite{CS} and \cite[Remark 2.4, p. 165]{S2}.
\item Furthermore, it is known that every homotopy sphere $\Sigma^{n}$ of dimension $n \geq 7$ has Gromoll filtration $\geq 2$, and this fact implies via \Cref{gromoll corollary} the inclusion $\{1, 2, n\} \subset S(\Sigma^{n})$ of (\ref{sgm on exotic spheres}).
Both of these facts follow from Cerf's work \cite{C} as pointed out in the introduction of \cite{CS} and in \cite[p. 279]{S1}.
\end{enumerate}
\end{remark}

\bibliographystyle{amsplain}

\begin{thebibliography}{10}

\bibitem {BdR} O. Burlet, G. de Rham, \textit{Sur certaines applications g\'{e}n\'{e}riques d'une vari\'{e}t\'{e} close \`{a} trois dimensions dans le plan},
Enseign. Math. \textbf{20} (1974) 275--292.

\bibitem{C} J. Cerf, \textit{La stratification naturelle des espaces de fonctions diff\'{e}rentiables r\'{e}elles et le th\'{e}or\`{e}me de la pseudo-isotopie},
Inst. Hautes \'{E}tudes Sci. Publ. Math. \textbf{39} (1970) 5--173.

\bibitem {CS} D. Crowley, T. Schick, \textit{The Gromoll filtration, KO-characteristic classes and metrics of positive scalar curvature},
Geometry and Topology \textbf{17} (2013) 1773--1789.

\bibitem{E} Y.M. Eliashberg, \textit{On singularities of folding type},
Math. USSR-Izv. \textbf{4} (1970) 1119--1134.

\bibitem{G} D. Gromoll, \textit{Differenzierbare Strukturen und Metriken positiver Kr\"{u}mmung auf Sph\"{a}ren},
Math. Ann. \textbf{164} (1966) 353--371.

\bibitem{H} A.E. Hatcher, \textit{A proof of the Smale conjecture, $\operatorname{Diff}(S^{3}) \simeq \operatorname{O}(4)$},
Ann. of Math. \textbf{117} (1983) 553-607.

\bibitem{KM} M. Kervaire, J.W. Milnor, \textit{Groups of homotopy spheres: I},
Ann. of Math. \textbf{77} (1963) 504--537.

\bibitem {M1} J.W. Milnor, \textit{Lectures on the h-cobordism theorem},
Math. Notes, Princeton Univ. Press, Princeton, NJ, 1965.

\bibitem {R} G. Reeb, \textit{Sur certains propri\'{e}t\'{e}s topologiques des vari\'{e}t\'{e}s feuillet\'{e}es},
Actualit\'{e}s Scientifiques et Industrielles 1183 (Hermann, Paris, 1993) 91--154.

\bibitem {S1} O. Saeki, \textit{Topology of special generic maps of manifolds into Euclidean spaces},
Topology Appl. \textbf{49} (1993) 265--293.

\bibitem {S2} O. Saeki, \textit{Topology of special generic maps into $\mathbb{R}^{3}$},
Workshop on Real and Complex Singularities (S\~{a}o Carlos, 1992), Mat. Contemp. \textbf{5} (1993) 161--186.

\bibitem{S3} O. Saeki, \textit{Topology of manifolds and global theory of singularities},
RIMS K\^{o}ky\^{u}roku Bessatsu \textbf{B55} (2016) 185--203.

\bibitem {W} M. Weiss, \textit{Pinching and concordance theory},
J. Differential Geometry \textbf{38} (1993) 387--416.

\bibitem {W1} D.J. Wrazidlo, \textit{Fold maps and positive quantum field theories},
dissertation, Heidelberg University (2017). \par
URL: \url{https://www.mathi.uni-heidelberg.de/~dwrazidlo/research_papers/thesis.pdf}

\end{thebibliography}

\end{document}